\documentclass[a4paper]{scrartcl}

\usepackage{amsfonts,amsthm,amssymb,amsmath}
\usepackage[latin1]{inputenc}
\usepackage[T1]{fontenc}
\usepackage{graphicx}
\usepackage{enumitem}
\usepackage{marvosym}

\newtheorem{thm}{Theorem}[section]
\newtheorem{lem}[thm]{Lemma}
\newtheorem{con}[thm]{Conjecture}

\newcommand{\bigO}{\mathcal O}
\newcommand{\sph}[1]{S_{v_0}\! \left( #1 \right)}
\newcommand{\ceil}[1]{\left \lceil #1 \right \rceil}
\newcommand{\floor}[1]{\left \lfloor #1 \right \rfloor}
\newcommand{\res}[1]{\rvert_{#1}}
\newcommand {\Aut}  {\ensuremath{\operatorname{Aut}}}
\newcommand {\id}  {\ensuremath{\operatorname{id}}}

\renewcommand{\qedsymbol}{\Rectsteel}

\begin{document}

\title{Distinguishing graphs with intermediate growth}
\author{Florian Lehner\thanks{The author acknowledges the support of the Austrian Science Fund (FWF), project W1230-N13.}}
\maketitle
\begin{abstract}
A graph $G$ is said to be 2-distinguishable if there is a 2-labeling of its vertices which is not preserved by any nontrivial automorphism of $G$. We show that every locally finite graph with infinite motion and growth at most $\bigO \left( 2^{(1-\varepsilon)\frac{\sqrt n}{2}}\right)$ is 2-distinguishable. Infinite motion means that every automorphism moves infinitely many vertices and growth refers to the cardinality of balls of radius $n$.
\end{abstract}

\section{Introduction} 
\label{sec:intro}

The distinguishing number of a graph $G$ is the smallest number $d$ such that there is a coloring of the vertices of $G$ with $d$ colors which is not preserved by any nontrivial automorphism of $G$. Since  it was introduced by Albertson and Collins \cite{MR1394549} numerous papers investigating the distinguishing number of finite and infinite graphs have been published \cite{MR2880672,MR2408368,MR2262268,MR2302543,MR2261819,MR1617449,smtuwa,MR2302536}. 
Furthermore several results concerning the distinguishing number could be extended to endomorphisms \cite{imkapile} and uncountable graphs \cite{cuimle} as well as other combinatorial structures \cite{MR2781021,MR2255126,MR2587751}.

In this paper we will focus on infinite locally finite graphs. More precisely we will verify a conjecture of Tucker \cite{MR2776826} for a class of infinite, locally finite graphs.

Before stating Tucker's conjecture and the main result of this paper we introduce some notions. The motion $m(\varphi)$ of an automorphism $\varphi$ of a graph $G$ is the number of vertices moved by $\varphi$. The motion of $G$ is the minimal motion of a nontrivial automorphism. Intuitively it is clear that large motion implies low distinguishing number. In the finite case this intuition can be quantified by Lemma \ref{lem:motion_orig} which was first published by Russel and Sundaram \cite{MR1617449}. In the case of infinite locally finite graphs the following connection of motion and distinguishing number has been conjectured by Tucker \cite{MR2776826}.

\begin{con}
\label{con:tucker}
Every infinite, connected locally finite graph with infinite motion is 2-distinguishable.
\end{con}

While this conjecture is still open in its full generality, it is known to be true for many classes of graphs including trees \cite{MR2302536}, tree-like graphs \cite{MR2302543} and graphs with countable automorphism group \cite{istw}. It is also known that the countable random graph has infinite motion as well as distinguishing number 2 \cite{MR2302543}.

From the 2-distinguishability of graphs with countable automorphism group it follows that all graphs with linear growth and infinite motion have distinguishing number 2. Growth here refers to the growth of balls with fixed center and increasing radius. Recently it has been shown by Cuno et al.\,\cite{cuimle} that the statement of Conjecture \ref{con:tucker} remains true if the growth is slightly superlinear, that is, for graphs of growth $o(\frac{n^2}{\log n})$. The main result of this paper is a further step in this direction:

\begin{thm}
\label{thm:growthmotion}
Let $G$ be a connected graph with infinite motion and growth $\bigO \! \left(2^{(1-\varepsilon) \frac {\sqrt n} 2}\right )$. Then $G$ is 2-distinguishable.
\end{thm}

The rest of this paper is structured as follows. In Sections \ref{sec:notions} and \ref{sec:auxiliary} we introduce all notions and tools needed for the proof of the main result. Since the proof of Theorem \ref{thm:growthmotion} is rather technical, Section \ref{sec:proof} contains a detailed outline of the main proof ideas prior to the actual proof. Finally, in Section \ref{sec:ends} we prove a result similar to Theorem \ref{thm:growthmotion} for graphs with countably many ends and outline a common extension of the two results.

\section{Notions and notations}
\label{sec:notions}

Throughout this paper $\log$ denotes the base 2 logarithm, $G$ denotes a connected locally finite graph with vertex set $V$ and edge set $E$, and $\Aut G$ denotes the set of automorphisms of $G$. All graph theoretical notions which are not explicitly defined will be taken from \cite{MR2159259}. 

For a vertex $v_0 \in V$ the \emph{sphere} $\sph{n}$ of radius $n$ with center $v_0$ is the set of all vertices $v$ such that $d(v_0,v) = n$, where $d$ denotes the natural geodesic metric on $G$. The \emph{ball} $B_{v_0}\! (n)$  is the set of all vertices $v$ such that $d(v, v_0) \leq n$. It is immediate that $B_{v_0} \! (n) = \bigcup_{k\leq n} \sph k$ and that this union is disjoint. A graph has growth $\bigO(f(n))$ if the cardinality of the ball of radius $n$ is $\bigO(f(n))$. Similarly we can define graphs of growth $o(f(n))$ and $\omega(f(n))$. Note that this does not depend on the choice of the basepoint $v_0$. By a \emph{graph of intermediate growth} we mean a graph whose growth is superpolynomial, but still not exponential, that is, the growth is $\omega(n^c)$ and $o(c^n)$ for every constant $c>1$.

A \emph{coloring of a set $S$} is a map $c\colon S \to C$ where $C$ is a finite set. We refer to the elements of $C$ as \emph{colors}. Mostly $C$ will be the set $\{\text{black},\text{white}\}$ in which case we speak of a \emph{2-coloring of $S$}. By a \emph{partial coloring of $S$} we mean a map $c' \colon S' \to C$ where $S' \subset S$. The set $S'$ is called the \emph{support of the partial coloring $c'$}. A \emph{(partial) coloring of the graph $G$} always refers to a (partial) coloring of the vertices of $G$, that is $S=V$.

Let $\varphi$ be a permutation of a set $S$. Throughout this paper one can usually think of $\varphi$ as a permutation of the vertex set of a graph $G$ induced by an automorphism $\varphi \in \Aut(G)$. We say that $\varphi$ \emph{preserves a coloring $c$} if $c(s)= c(\varphi(s))$ holds for every $s \in S$. The permutation $\varphi$ \emph{preserves a partial coloring $c'$} if we can extend $c'$ to a coloring $c$ which is preserved by $\varphi$. If $\varphi$ does not preserve a (partial) coloring $c$ we say that $c$ \emph{breaks} $\varphi$. 

Let $S$ be a set and let $A$ be a set of permutations of $S$. We say that a (partial) coloring $c$ of $S$ is \emph{$A$-distinguishing} if every permutation in $A$ is broken by $c$. A coloring of a graph $G$ is \emph{distinguishing} if it is $A$-distinguishing for $A = \Aut(G) \setminus \{\id\}$. A graph $G$ is \emph{$d$-distinguishable} if there is a distinguishing coloring with $d$ colors. The \emph{distinguishing number} of $G$ is the smallest number $d$ such that $G$ is $d$-distinguishable.

Let $\varphi$ be a permutation of a set $S$. The \emph{motion} $m(\varphi)$ is the number of elements of $S$ which are not fixed by $\varphi$. If $\varphi$ fixes $S' \subset S$ as a set we define the \emph{restriction} $\varphi \res{S'}$ of $\varphi$ to $S'$ to be the permutation which $\varphi$ induces on $S'$. The \emph{restricted motion} $m(\varphi)\res{S'}$ is the number of elements in $S'$ which are not fixed by $\varphi$. For a set $A$ of permutations we define the \emph{restriction} $A \res{S'}$ to be the set of all distinct permutations $\varphi \res{S'}$ where $\varphi \in A$. The \emph{motion} $m(A)$ is defined as the minimal motion of all permutations in $A$. Analogously define the \emph{restricted motion} $m(A) \res{S'}$. The \emph{motion of a graph $G$} is the motion of $\Aut(G) \setminus \{\id\}$. 

We could define the minimum of the empty set as $\infty$, so a graph with no nontrivial automorphism has infinite motion by default. However, graphs with no nontrivial automorphisms are trivially 2-distinguishable, hence they will not play any role in the sequel.

\section{Auxiliary results}
\label{sec:auxiliary}

In this section we would like to present some results that will be used in the proof of Theorem \ref{thm:growthmotion}.

The following lemma from \cite{cuimle} basically states that by coloring a small fraction of the vertices of $G$ we can break all automorphisms that move a given vertex.
\begin{lem}
\label{lem:fixroot}
Let $G =(V,E)$ be an infinite, locally finite, connected graph with infinite motion, $v_0 \in V$. For every $\delta >0$ there is a partial coloring $c$ of the vertices of $G$ with the following properties:
\begin{enumerate}
\item $c$ is $A$-distinguishing for $A= \{\varphi \in \Aut G  \mid \varphi(v_0) \neq v_0 \}$.
\item There is $k_0$ such that less than $\delta k$ of the spheres $\sph{m+1}, \ldots , \sph {m+k}$ are colored for every $k > k_0$ and every $m \in \mathbb N$.{\hfill \qedsymbol}
\end{enumerate}
\end{lem}
The proof of the above lemma is constructive. Readers interested in the details of the construction are referred to \cite{cuimle}.

Since we now can focus on automorphisms that fix a given vertex $v_0$ it would be convenient to know some properties of these automorphisms. The following lemma which can be found in \cite{cuimle} summarizes several useful properties.
\begin{lem}
\label{lem:higher}
Let $G =(V,E)$ be an infinite, locally finite, connected graph with infinite motion. Let $\varphi, \varphi' \in \Aut G$ and assume that there is $v_0 \in V$ such that $\varphi(v_0) = v_0$. Then
\begin{enumerate}
\item for every $i \in \mathbb N$ it holds that $\varphi$ fixes $\sph i$ as a set,
\item $m(\varphi) \res{\sph i} > 0$ implies that $\forall j > i \colon m(\varphi) \res{\sph j} > 0$,
\item $\varphi \res{\sph i} = \varphi' \res{\sph i}$ if and only if $\varphi \res{B_{v_0}\!(i)} = \varphi' \res{B_{v_0}\!(i)}$. {\hfill \qedsymbol}
\end{enumerate}
\end{lem}
The proof of this lemma is easy and straightforward. The first property follows from the fact that an automorphism is an isometry. The second part is proved by contradiction using the fact that $G$ has infinite motion. In the third property one implication is trivial while the other one can be derived from the second property.

Next, the proof of Theorem \ref{thm:growthmotion} relies heavily on the following lemma which was first proved by Russel and Sundaram \cite{MR1617449}. Although the result we need is somewhat stronger than the result in \cite{MR1617449}, it can be shown in a very similar way. For the convenience of the reader we include a short proof.

\begin{lem}
\label{lem:motion}
Let $S$ be a set and let $A$ be a set of permutations of $S$. Let $S' \subset S$ be a finite set that is fixed by every $\varphi \in A$. If
\[
	m(A)\res{S'} > 2 \log \left \vert A \res{S'} \right \vert,
\]
then there is an $A$-distinguishing partial 2-coloring of $S$ with support $S'$.
\end{lem}

\begin{proof} 
Denote by $C$ the set of all partial 2-colorings with support $S'$. Double counting yields
\[
	\sum _{\varphi \in A \res{S'}} \left \vert \{ c \in C \mid \varphi \text{ preserves } c\} \right \vert
	=
	\sum _{c \in C} \left \vert \{ \varphi \in A \res{S'} \mid \varphi \text{ preserves } c\} \right \vert.
\]
If the latter sum is less than $2^{\vert S' \vert}$ then there has to be at least one 2-coloring which is not preserved by any $\varphi \in A \res{S'}$.

To estimate the first sum note that the number of 2-colorings that are preserved by a given permutation $\varphi \in A\res{S'}$ is equal to $2^p$ where $p$ is the number of cycles in the permutation. Clearly there are $\vert S' \vert - m(\varphi)\res{S'}$ singleton cycles while the rest of the elements of $S'$ is partitioned in at most $\frac{m(\varphi)\res{S'}}{2}$ cycles of length at least $2$. So we have
\[
	p \leq \vert S' \vert -\frac{m(A)\res{S'}}{2}
\]
and we can estimate the sum 
\[
	\sum _{\varphi \in A\res{S'}} \left \vert \{ c \in C \mid \varphi \text{ preserves } c\} \right \vert
	\leq \sum _{\varphi \in A\res{S'}} 2^{\vert S' \vert -\frac{m(A)\res{S'}}{2}}
	< \vert A\res{S'} \vert \,  2^{\vert S' \vert} \, 2^{- \log \vert A\res{S'} \vert}
	= 2^{\vert S' \vert} \qedhere
\]

\end{proof}

Putting $S=S'=V$ and $A = \Aut G \setminus \{\id\}$ for a finite graph $G=(V,E)$ we obtain the formulation of \cite{MR1617449} as a corollary:
\begin{lem}[Motion Lemma]
\label{lem:motion_orig}
Let $G$ be a finite graph. If the motion of $G$ is larger than $2 \log \vert \Aut G \vert$, then $G$ is $2$-distinguishable.{\hfill \qedsymbol}
\end{lem}

Finally we present a result which will be useful in the proof of the extension of Theorem \ref{thm:growthmotion} in Section \ref{sec:ends}. 

\begin{lem}
\label{lem:infcomp}
Let $G$ be a graph with infinite motion, let $\varphi \in \Aut (G)$ and denote by $V_{\text{fix}} \subseteq V$ the set of fixed points of $\varphi$. Then the graph $G - V_{\text{fix}}$, which is obtained from $G$ by removing $V_{\text{fix}}$ and all incident edges, has only infinite components.
\end{lem}

\begin{proof}
If there was a finite component $C$ then we could define an automorphism $\varphi'$ which coincides with $\varphi$ on this component and fixes every vertex $v \notin C$. This automorphism is easily seen to have finite motion, which contradicts $G$ having infinite motion.
\end{proof}

Notice that in the above lemma we do not require that the graph is connected  or locally finite, in fact it may even be uncountable. However, if $G$ is locally finite then Lemma \ref{lem:infcomp} implies that for every automorphism $\varphi$ of $G$ with infinite motion there is a ray which contains no fixed point of $\varphi$. If furthermore $G$ is connected and the automorphism $\varphi$ has at least one fixed point we get the following result.

\begin{lem}
\label{lem:disjointray}
Let $G$ be a connected locally finite graph with infinite motion, let $\varphi \in \Aut (G)$ and assume that there is a vertex $v \in V$ such that $\varphi (v) = v$. Then every component of $G - V_{\text{fix}}$ contains a ray $\gamma$ which is mapped to a disjoint ray $\gamma'$.
\end{lem}

\begin{proof}
Let $C$ be a component of $G-V_{\text{fix}}$. First notice that there must be a ray in $C$ since $G$ is locally finite and $C$ is infinite by Lemma \ref{lem:infcomp}.

Any two vertices in $C$ are connected by a path which does not use any vertex in $V_{\text{fix}}$. Clearly the image of such a path is again a path which does not contain any vertex in $V_{\text{fix}}$. Hence if some vertex in $C$ has an image outside of $C$ then so do all vertices of $C$. So in this case each ray in $C$ is mapped to a disjoint ray.

Now assume that $C$ is fixed by $\varphi$. Choose a fixed point $v_0$ of $\varphi$ which is adjacent to some vertex in $C$. Note that such a vertex $v_0$ must exist because there is a path connecting $C$ to $v$. Consider the graph $G'$ which is obtained from $C$ by adding $v_0$ and all edges between $v_0$ and $C$.

Using breadth-first-search construct a spanning tree $T$ of $G'$ with root $v_0$. Note that $\varphi$ acts on $G'$ as an automorphism. Since every  automorphism is an isometry it holds that for every $w \in C$ the vertices $w$ and $\varphi (w)$ have the same distance from $v_0$ in $G'$. Thus they also have the same distance from $v_0$ in $T$. 

Choose a ray $\gamma$ in $T$ which starts at a neighbour of $v_0$ but does not use $v_0$. Then all vertices in $\gamma$ have different distances from $v_0$. Since no $w \in \gamma$ is mapped to itself it is clear that $\gamma$ must be mapped to a disjoint ray.
\end{proof}

\section{Proof of the main result}
\label{sec:proof}

Before proving Theorem \ref{thm:growthmotion} we would like to provide a sketch of the proof to explain the main ideas.

By Lemma \ref{lem:fixroot} we can assume that there is a vertex $v_0$ which is fixed by every automorphism that we still need to break. By Lemma \ref{lem:higher} every such automorphism fixes every sphere $\sph{i}$ as a set, so it makes sense to speak of restricted motion.

Now assume that we would like to break the set $A$ of all automorphisms that act nontrivially on $\sph{m}$. We know by Lemma \ref{lem:higher} that every $\varphi \in A$ also acts nontrivially on every higher sphere. We choose $k$ ``large enough'' (we will specify later, how large it must be) and split up the set of spheres $\sph{m+1}, \ldots, \sph{m+k}$ in some small sets $P_i$ and a remainder set $P_r$. Following a suggestion of Imrich we partition $A$ into several sets $A_i$ of automorphisms whose motion on one of the spheres $\sph{m+1}, \ldots, \sph{m+k}$ is small and a remainder set $A_r$ in which every automorphism has large restricted motion on each of those spheres.

Since the cardinality of the sets $A_i$ is small we can apply Lemma \ref{lem:motion} to break all of $A_i$ by a coloring of $P_i$ although the motion of the elements of $A_i$ may be small. Similarly we can break all automorphisms in $A_r$ by a coloring of $P_r$ since the motion is large.

Having broken all automorphisms in $A$ we proceed inductively breaking all automorphisms which act nontrivially on $\sph {m+k}$. In the limit we obtain a coloring which breaks every nontrivial automorphism because every such automorphism has to act nontrivially on some sphere.

Now let us turn to a detailed proof of Theorem \ref{thm:growthmotion}.

\begin{proof}[Proof of Theorem \ref{thm:growthmotion}]
First of all apply Lemma \ref{lem:fixroot} with $\delta = \frac \varepsilon 2$ and an arbitrarily chosen vertex $v_0$. This gives a coloring of a small fraction of the spheres which breaks all automorphisms that do not fix $v_0$. Recall from the statement of the theorem that $\delta > 0$ is arbitrary, hence we can assume $0 < \varepsilon < 1$. As mentioned before, every unbroken automorphism must fix every sphere $S_{v_0}(i)$ as a set.

Now assume that all spheres up to $\sph{m}$ have already been colored while $\sph{m+1}$ is still uncolored. We know that there is a constant $c$ such that for large $n$ it holds that
\[
\vert B_{v_0} \! (n) \vert \leq c \, 2^{(1-\varepsilon) \frac{\sqrt{n}}{2}}.
\]
By increasing the constant $c$ we can guarantee that this inequality holds for every $n$. Next notice that 
\[
\sqrt{m+k} \leq \sqrt m + \sqrt k
\]
and hence

\begin{align*}
\vert B_{v_0} \! (m+k) \vert 
&\leq c \, 2^{(1-\varepsilon) \frac{\sqrt{m+k}}{2}} \\
&\leq c \, 2^{(1-\varepsilon) \sqrt{\frac{m}{2}}} \, 2^{(1-\varepsilon) \frac{\sqrt{k}}{2}} \\
&= \tilde c \, 2^{(1-\varepsilon) \frac{\sqrt{k}}{2}}
\end{align*}
where $\tilde c $ depends on $c$ and $m$. Note that this clearly implies that for every $i \leq m+k$ we have
\begin{equation}
\label{ineq:spheresize} \left \vert \sph{i} \right \vert < \tilde c \, 2^{(1-\varepsilon) \frac{\sqrt{k}}{2}}.
\end{equation}

Now choose $k$ larger than $k_0$ given by Lemma \ref{lem:fixroot} and large enough that each of the following inequalities holds:
\begin{align}
\label{ineq:logc} \log {\tilde c} &< \frac {\varepsilon \sqrt{k}}8 ,\\
\label{ineq:logk} \log  k &< \frac {\varepsilon \sqrt{k}}8 ,\\
\label{ineq:sqrtk} 4 \sqrt k &< \frac 1 2 \, \varepsilon \left( 1- \frac \varepsilon 2\right ) k ,\\
\label{ineq:csqrtk} \tilde c \, \frac {\sqrt {k}} 2 &< \frac \varepsilon 4\, k.
\end{align}

Notice that these inequalities are by no means independent. For example it is easy to see that if $\tilde c$ is large (which usually will be the case) then \eqref{ineq:csqrtk} implies \eqref{ineq:logc} and \eqref{ineq:sqrtk}. However, we will need all four inequalities in the proof so we might as well explicitly require them.

Next consider the spheres $\sph{m+1}, \ldots, \sph{m+k}$. We know that at least $(1-\varepsilon) k$ of these spheres are still uncolored, denote those spheres by $S_1, \ldots, S_l$ ordered in a way that $S_i$ lies closer to $v_0$ than $S_{i+1}$.

Define 
\begin{align*}
\kappa &= \ceil{2 \sqrt {k} \left ( 1-\frac{\varepsilon}{2}\right)},\\
r &= \ceil{(1-\varepsilon)\frac {\sqrt {k}} 2} +1.
\end{align*}
For $1 \leq i \leq r-1$ let $P_i$ be the set of vertices contained in $S_{(i-1) \kappa +1}, \ldots, S_{i\kappa}$. The vertices contained in $S_{(r-1)\kappa}, \ldots, S_l$ are collected in the set $P_r$. Obviously $P_i$ contains $\kappa$ spheres for $i<r$. Let us briefly check how many spheres there are in $P_r$:
\begin{align*}
l- \sum _{i=1}^{r-1} \kappa 
& \geq \left(1- \frac \varepsilon 2 \right) k - \kappa (r-1) \\
& \geq \left(1- \frac \varepsilon 2 \right) k - \left (2 \sqrt {k} \left ( 1-\frac{\varepsilon}{2}\right) +1 \right ) \left ( (1-\varepsilon)\frac {\sqrt {k}} 2 +1 \right )\\
& = \left(1- \frac \varepsilon 2 \right) k - \left (\left ( 1-\frac{\varepsilon}{2}\right)(1-\varepsilon)k + 2 \sqrt {k} \left ( 1-\frac{\varepsilon}{2}\right) + (1-\varepsilon) \frac {\sqrt {k}} 2 + 1\right)\\
& \geq \varepsilon \left(1- \frac \varepsilon 2 \right) k - \left (2  + \frac 1 2 + 1\right) \sqrt k \\
& \geq \varepsilon \left(1- \frac \varepsilon 2 \right) k - 4 \sqrt k \\
& > \frac \varepsilon 2 \left(1- \frac \varepsilon 2 \right) k.
\end{align*}
where the last inequality follows from \eqref{ineq:sqrtk}. So altogether we have partitioned the spheres $S_1, \ldots, S_l$ into $r-1$ sets of $\kappa$ spheres and a set of  more than $\frac \varepsilon 2 \left(1- \frac \varepsilon 2 \right) k$ spheres.

Next we would like to partition the set $A$ of automorphisms that act nontrivially on $\sph m$ into sets $A_i$ such that
\[
m(A_i)\res{P_i} > 2 \log \left \lvert A_i \res{P_i} \right \rvert.
\]
This will enable us to apply Lemma \ref{lem:motion} to break all permutations in $A_i$ by a coloring of the set $P_i$. If we color every $P_i$ according to this coloring we obtain a partial coloring of $G$ which breaks every automorphism that acts nontrivially on $\sph{m}$.

In order to define the sets $A_i$ let 
\[
A_i'  =  \left \{ \varphi \in A \mid \exists S_j \subseteq P_{i'}, r \geq i' >i \colon m(\varphi)\res{S_j}\leq 2^i  \right \}.
\] 
In words, $A_i'$ contains all automorphisms $\varphi$ which move at most $2^i$ vertices in some sphere that lies above $P_i$. Define $A_i = A_i' \setminus A_{i-1}'$ for $1 \leq i<r$ and $A_r = A \setminus \bigcup_{i<r} A_i$ where $A_0' = \emptyset$. Notice that for $S_j \subseteq P_i$ and $\varphi \in A_i$ it holds that
\[
m(\varphi)\res{S_j} > 2^{i-1}
\]
because otherwise $\varphi$ would be contained in $A_{i-1}'$.

Now that we have partitioned both the uncolored spheres and the automorphisms that we wish to break, let us check if we can apply Lemma \ref{lem:motion} to the sets $A_i$ and $P_i$.  We establish an upper bound for $\vert  A_i\res{P_i} \vert$ and a lower bound for the motion an automorphism $\varphi \in A_i$ has on $P_i$.

Clearly $\vert A_i \res{P_i} \vert \leq \vert A_i' \res{P_i} \vert$. First we consider the case $i < r$. The case $i = r$ will be treated later. 

To estimate the cardinality of $ A_i'\res{P_i}$ notice that by Lemma \ref{lem:higher} every permutation in a sphere $S_j$ in $P_{i+1}, \ldots, P_r$ induces a unique permutation on $P_i$. Hence we only need to count the permutations which move at most $2^i$ elements in one of these spheres. Since there are at most $k$ such spheres and the cardinality of each of them is bounded from above by $\floor{\tilde c 2^{(1-\varepsilon)\frac{\sqrt k}{2}}}$ according to \eqref{ineq:spheresize}, we get the following estimate:
\begin{align*}
	\vert  A_i \res{P_i} \vert 
	&\leq k \binom{\floor{\tilde c 2^{(1-\varepsilon) \frac {\sqrt{k}} 2}}}{2^i} \left( 2^i \right) ! \\
	&\leq k \frac{\left( \tilde c 2^{(1-\varepsilon) \frac {\sqrt{k}} 2}\right)^{2^i}}{\left( 2^i \right) !} \left( 2^i \right) ! \\
	&=2^{\log k + \left( \log {\tilde c} + (1-\varepsilon) \frac {\sqrt{k}} 2\right) 2^i} \\
	&\leq 2^{\left( \log k +  \log {\tilde c} + (1-\varepsilon) \frac {\sqrt{k}} 2\right) 2^i} \\
	&< 2^{\left (1-\frac{\varepsilon} 2 \right ) \frac {\sqrt{k}} 2 2^i}.
\end{align*}
The last inequality follows from \eqref{ineq:logc} and \eqref{ineq:logk}.

In order to estimate the motion of $\varphi \in A_i$ on $P_i$ recall that there are $\kappa$ spheres in $P_i$ and $\varphi$ moves at least $2^{i-1}$ elements in each of the spheres. Hence we get the following inequality.
\begin{align*}
m(\varphi) \geq \kappa 2^{i-1} \geq 2\sqrt{k} \left( 1- \frac \varepsilon 2 \right ) 2^{i-1}.
\end{align*}

If we combine the two estimates we obtain
\begin{align*}
\frac{m(A_i) \res{P_i}}2 \geq \sqrt{k} \left( 1- \frac \varepsilon 2 \right ) 2^{i-1} > \log{\vert A_i \res{P_i} \vert}
\end{align*}
which is exactly the inequality in the condition of Lemma \ref{lem:motion} for $1 \leq i < r$. So we can apply the motion Lemma in order to break all of $A_i$ by a suitable coloring of $P_i$.

Finally we need to verify that the inequality also holds true for $A_r$ and $P_r$. By Lemma \ref{lem:higher} the number of permutations in $ A_r \res{P_r}$ is bounded by the number of permutations of $S_{m+k}$, that is
\[
\vert A_r\res{P_r} \vert 
\leq \vert S_{m+k} \vert ! 
\leq \floor{\tilde c 2^{(1-\varepsilon) \frac {\sqrt{k}} 2}} !
\leq 2^{\left( \log \tilde c + (1-\varepsilon) \frac {\sqrt{k}} 2 \right ) \tilde c 2^{(1-\varepsilon) \frac {\sqrt{k}} 2}}
< 2^{ \tilde c  (1- \frac \varepsilon 2) \frac {\sqrt{k}} 2 2^{(1-\varepsilon) \frac {\sqrt{k}} 2}}.
\]
The last inequality easily follows from \eqref{ineq:logc}.

In order to estimate the motion notice that every $\varphi \in A_r$ moves at least $2^{r-1}$ vertices in each sphere in $P_r$. Since there are more than $\frac \varepsilon 2 \left(1- \frac \varepsilon 2 \right) k$ spheres in $P_r$ we get
\[
 m(\varphi) > \frac \varepsilon 2 \left(1- \frac \varepsilon 2 \right) k 2^{r-1} 
 \geq \frac \varepsilon 2 \left(1- \frac \varepsilon 2 \right) k 2^{(1-\varepsilon) \frac{\sqrt k} 2}.
\]
Putting these estimates together we obtain the inequality in the condition of Lemma \ref{lem:motion}
\[
\frac{m(A_r)\res{P_r}}{2} > \frac \varepsilon 4 \left(1- \frac \varepsilon 2 \right) k 2^{(1-\varepsilon) \frac{\sqrt k} 2}
>  \tilde c  \left(1- \frac \varepsilon 2 \right) \frac {\sqrt{k}} 2 2^{(1-\varepsilon) \frac {\sqrt{k}} 2}
> \log \vert  A_r \res{P_r}\vert
\]
where the middle inequality is a direct consequence of \eqref{ineq:csqrtk}. This proves that we can apply Lemma \ref{lem:motion} to find a 2-coloring of $P_r$ which breaks every automorphism in $A_r$.

So we have shown that we can break all of $A$ by a 2-coloring of the part of the spheres $\sph {m+1}, \ldots, \sph{m+k}$ that has not been colored when we applied Lemma \ref{lem:fixroot}. 

Iteratively proceed by breaking all automorphisms that fix $v_0$ and act nontrivially on $\sph{m+k}$. Clearly in the limit this yields a coloring that breaks every nontrivial automorphism. Simply note that every nontrivial automorphism that fixes $v_0$ has to act nontrivially on some sphere $\sph n$ and thus also on every higher sphere.
\end{proof}

The reader may have noticed that in the proof we have only used that the size of the spheres is bounded by $2^{(1-\varepsilon)\frac {\sqrt n} 2}$. Since the ball $B_{v_0}\! (n)$ is the union of all spheres of radius at most $n$ one might wonder if the same proof gives a better bound on the growth of the graph. This is however not the case since
\[
	\sum_{k=1}^n 2^{(1-\varepsilon)\frac {\sqrt k} 2} 
	\leq c \int_{0}^n 2^{(1-\varepsilon)\frac {\sqrt x} 2} \, dx 
	\leq c' \sqrt n 2^{(1-\varepsilon)\frac {\sqrt n} 2} 
	\leq 2^{(1-\frac\varepsilon 2)\frac {\sqrt n} 2}
\]
for large values of $n$ and suitable constants $c, c'$. So if the sphere of radius $n$ has size $\bigO \left( 2^{(1-\varepsilon)\frac {\sqrt n} 2}\right)$, then the same holds true for the ball of radius $n$ with a slightly different $\varepsilon$.

\section{Growth of ends}
\label{sec:ends}

In this section we would like to outline a possible generalisation of Theorem \ref{thm:growthmotion} to graphs with countably many ends. For readers not familiar with the notion of ends we refer to \cite{MR2159259} for an accessible introduction to the topic. Before stating the extension, however, we need to introduce the notion of the growth of an end.

If we consider an end $\omega$ of a graph and a base vertex $v_0$ we can define the set 
\[
S_{v_0}^\omega(n) = \left \lbrace v \in \sph n \mid v \text{ lies in the same component of } G \setminus B_{v_0}(n-1) \text{ as } \omega \right \rbrace.
\]
An end $\omega$ has \emph{growth $\bigO(f(n))$} if the cardinality of $S_{v_0}^\omega(n)$ is $\bigO(f(n))$. It is worth noting that---just like the growth of a graph---the growth of an end does not depend on the choice of the basepoint $v_0$.

In \cite{cuimle} the result that every graph of growth $o(\frac{n^2}{\log n})$ is 2-distinguishable is extended to graphs with countably many ends where no end grows faster than $o(\frac{n}{\log n})$. Note that the definition of growth of ends uses spheres rather than balls, so for one ended graphs the two results coincide.

Since the proof of our main result is somewhat similar to the proof of the corresponding result in \cite{cuimle}, it is not surprising that we can use the same arguments to extend it to graphs where the growth of every end is $\bigO(2^{(1-\varepsilon)\frac {\sqrt n} 2})$.

\begin{thm}
\label{thm:ends}
Let $G$ be a connected graph with countably many ends each of which has growth $\bigO(2^{(1-\varepsilon)\frac {\sqrt n} 2})$ for the same fixed $\varepsilon$. If $G$ has infinite motion then $G$ is 2-distinguishable.
\end{thm}

\begin{proof}
First of all---just as in the proof of Theorem \ref{thm:growthmotion}---find a partial coloring $c$ which fixes a vertex $v_0$. The only difference is, that we choose $\delta = \frac{\varepsilon}{4}$ rather than $\delta = \frac{\varepsilon}{2}$.

The rest of the proof divides into two steps. First we extend $c$ to a partial coloring that breaks every automorphism of $G$ which does not fix the set of ends of $G$ pointwise, still leaving a large fraction of the vertices uncolored. Then we use the same argument as in the proof of Theorem \ref{thm:growthmotion} in order to color the rest of the vertices such that the remaining automorphisms are broken.

For the first step choose an increasing sequence $n_i$ such that the spheres $\sph{n_i}$ are still uncolored and $n_i-n_{i-1}>\frac 4 \varepsilon$. Consider the set of spheres $\sph{n_i}$. We wish to color those spheres such that every automorphism that fixes $v_0$ and preserves the coloring also fixes every end of $G$. Notice that after coloring these spheres the fraction of uncolored spheres will still be at least $1- \frac{\varepsilon}{2}$.

It is not hard to see that the sets $S_{v_0}^\omega\left ( n_i\right )$ carry a rooted tree structure. Consider $v_0$, the root,  which is connected by an edge to every $S_{v_0}^\omega(n_1)$. Draw an edge from $S_{v_0}^\omega\left ( n_{i-1} \right )$ to $S_{v_0}^\omega\left ( n_i \right )$. To see that this is indeed a tree just notice that if $S_{v_0}^{\omega_1}(n)=S_{v_0}^{\omega_2}(n)$, then $S_{v_0}^{\omega_1}(m)=S_{v_0}^{\omega_2}(m)$ for every $m < n$, so there cannot be any circles.

Next, notice that every automorphism $\phi \in \Aut (G)$ that fixes $v_0$ but does not fix all ends also acts as a nontrivial automorphism on this rooted tree. By \cite{MR2302536} the distinguishing number of infinite leafless trees is at most $2$, therefore it is possible to 2-color the sets $S_{v_0}^\omega\left ( n_i \right )$ such that every such automorphism is broken. It is also worth noting that so far we did not use the countability of the end space of $G$, nor did we use the growth condition on the ends.

For the second step of the proof let us check which automorphisms of $G$ have not yet been broken. Denote the set of such automorphisms by $A$. We already know that every $\varphi \in A$ must fix $v_0$ as well as every end of $G$. Lemma \ref{lem:disjointray} implies that every automorphism of $G$ moves some ray of $G$ into a disjoint ray. Hence every automorphism in $A$ permutes some rays which belong to the same end $\omega$.

For an end $\omega$ of $G$ let $A^{\omega}$ be the set of permutations in $A$ which move some rays in $\omega$. Note that these sets are not necessarily disjoint but their union is all of $A$. Also notice that every automorphism $\varphi \in A^{\omega}$ acts nontrivially on every $S_{v_0}^{\omega}(n)$ from some index $n_0$ on.

Furthermore, let $(\omega_i)_{i \in {\mathbb N}}$ be an enumeration of the ends of $G$. Choose a function $f\colon {\mathbb N} \to {\mathbb N}$ such that $f^{-1}(i)$ is infinite for every natural number $i$. Assume that all spheres up to $S_{v_0}(m)$ have been colored in the first $i-1$ steps. In the $i$-th step we would like to color some more spheres in order to break all automorphisms in $A^{\omega_{f(i)}}$ that act nontrivially on $S_{v_0}^{\omega_{f(i)}}(n)$ for every $n>m$. Since we only colored an $\frac{\varepsilon}{2}$-fraction of all spheres so far, this can be achieved by exactly the  same arguments as in the proof of Theorem~\ref{thm:growthmotion}.

As we already mentioned, every automorphism that was not broken in the first step acts nontrivially on the rays of some end. Since, in the procedure described above, every end is considered infinitely often, it is clear that every such automorphism will eventually be broken. This completes the proof.
\end{proof}

The same proof still works if we can partition the (possibly uncountably many) ends into countably many classes such that the combined growth of all ends contained in a class $\mathcal C$ is $\bigO(2^{(1-\varepsilon)\frac {\sqrt n} 2})$. By combined growth we simply mean the growth of the cardinality of the sets
\[
	S_{v_0}^{\mathcal C}(n) = \bigcup_{\omega \in \mathcal C}S_{v_0}^\omega(n).
\]
This can be seen as a generalisation of both Theorem \ref{thm:growthmotion} (all ends in the same class) and Theorem \ref{thm:ends} (every end has its own class).

\section*{Acknowledgements}
The author would like to thank Wilfried Imrich and Johannes Cuno for reading earlier versions of this manuscript as well as making many useful suggestions.

\bibliography{sources}
\bibliographystyle{abbrv}
\end{document}